\documentclass{amsart}

\usepackage{amsmath, amsthm, amssymb}
\usepackage{enumerate}
\usepackage{verbatim}
\usepackage{esint}

\allowdisplaybreaks[3]

\numberwithin{equation}{section}


\email{guydavid@math.nyu.edu}
\address{Guy C. David\\ Courant Institute of Mathematical Sciences\\ New York University\\ 251 Mercer Street\\ New York, NY 10012 }

\title{Lusin-type theorems for Cheeger derivatives on metric measure spaces}
\author{Guy C. David}
\date{January 25, 2015}
\begin{document}

\begin{abstract}
A theorem of Lusin states that every Borel function on $\mathbb{R}$ is equal almost everywhere to the derivative of a continuous function. This result was later generalized to $\mathbb{R}^n$ in works of Alberti and Moonens-Pfeffer. In this note, we prove direct analogs of these results on a large class of metric measure spaces, those with doubling measures and Poincar\'e inequalities, which admit a form of differentiation by a famous theorem of Cheeger.
\end{abstract}

\maketitle
\newtheorem{thm}[equation]{Theorem}
\newtheorem{cor}[equation]{Corollary}
\newtheorem{claim}[equation]{Claim}
\newtheorem{lemma}[equation]{Lemma}
\newtheorem{prop}[equation]{Proposition}

\theoremstyle{definition}
\newtheorem{definition}[equation]{Definition}
\newcommand{\RR}{\mathbb{R}}
\newcommand{\LIP}{\textnormal{LIP}}
\newcommand{\dist}{\textnormal{dist}}
\newcommand{\diam}{\textnormal{diam }}
\newcommand{\Lip}{\text{\textnormal{Lip}}}

\section{Introduction}\label{sectionintro}

A classical theorem of Lusin \cite{Lu17} states that for every Borel function $f$ on $\RR$, there is a continuous function $u$ on $\RR$ that is differentiable almost everywhere with derivative equal to $f$.

In \cite{Al91}, Alberti gave a related result in higher dimensions. He proved the following theorem, in which $|\cdot|$ denotes Lebesgue measure and $Du$ denotes the standard Euclidean derivative of $u$.
\begin{thm}[\cite{Al91}, Theorem 1]\label{albertioriginal}
Let $\Omega\subset \RR^k$ be open with $|\Omega|<\infty$, and let $f\colon \Omega\rightarrow\RR^k$ be a Borel function. Then for every $\epsilon>0$, there exist an open set $A\subset \Omega$ and a function $u\in C^1_0(\Omega)$ such that
\begin{enumerate}[\normalfont (a)]
\item $|A|\leq \epsilon |\Omega|$,
\item $f=Du$ on $\Omega\setminus A$, and
\item $\|Du\|_p \leq C\epsilon^{\frac{1}{p} - 1}\|f\|_p$ for all $p\in [1,\infty]$.
\end{enumerate}
Here $C>0$ is a constant that depends only on $k$.
\end{thm}
In other words, Alberti showed that it is possible to arbitrarily prescribe the gradient of a $C^1_0$ function $u$ on $\Omega\subset\RR^k$ off of a set of arbitrarily small measure, with quantitative control on all $L^p$ norms of $Du$.

Moonens and Pfeffer \cite{MP08} applied Alberti's result to show a more direct analog of the Lusin theorem in higher dimensions:
\begin{thm}[\cite{MP08}, Theorem 1.3]\label{mporiginal}
Let $\Omega\subset \RR^k$ be an open set and let $f\colon \Omega\rightarrow \RR^k$ be measurable. Then for any $\epsilon>0$, there is an almost everywhere differentiable function $u\in C(\RR^k)$ such that
\begin{enumerate}[\normalfont (a)]
\item $\|u\|_\infty \leq \epsilon$ and $\{u\neq 0\}\subset \Omega$,
\item $Du = f$ almost everywhere in $\Omega$, and
\item $Df = 0$ everywhere in $\RR^k \setminus \Omega$. 
\end{enumerate}
\end{thm}

These ``Lusin-type'' results for derivatives in Euclidean space have applications to integral functionals on Sobolev spaces \cite{Al91}, to the construction of horizontal surfaces in the Heisenberg group (\cite{Ba03, HM13}) and in the analysis of charges and normal currents \cite{MP08}. In addition, we remark briefly that the results of Alberti and Moonens-Pfeffer have been generalized to higher order derivatives on Euclidean space in the work of Francos \cite{Fr12} and Haj\l asz-Mirra \cite{HM13}, though we do not pursue those lines here.

The purpose of this note is to extend the results of Alberti and Moonens-Pfeffer, in a suitable sense, to a class of metric measure spaces on which differentiation is defined.

In his seminal 1999 paper, Cheeger \cite{Ch99} defined (without using this name) the notion of a ``measurable differentiable structure'' for a metric measure space. Cheeger showed that a large class of spaces, the so-called PI spaces, possess such a structure. A differentiable structure endows a metric measure space with a notion of differentiation and a version of Rademacher's theorem: every Lipschitz function is differentiable almost everywhere with respect to the structure.

The class of PI spaces includes Euclidean spaces, all Carnot groups (such as the Heisenberg group), and a host of more exotic examples like those of \cite{BP99}, \cite{La00}, and \cite{CK13_PI}.

We prove the following two analogs of the results of Alberti and Moonens-Pfeffer for PI spaces. All the definitions are given in Section \ref{sectiondefs} below.

\begin{thm}\label{alberti}
Let $(X,d,\mu)$ be a PI space and let $\{(U_j, \phi_j\colon X\rightarrow\RR^{k_j})\}_{j\in J}$ be a measurable differentiable structure on $X$. Then there are constants $C,\eta>0$ with the following property:

Let $\Omega \subset X$ be open with $\mu(\Omega)<\infty$ and let $\{f_j\colon U_j \cap \Omega \rightarrow \RR^{k_j}\}_{j\in J}$ be a collection of Borel functions. Then for all $\epsilon>0$ there is an open set $A\subset \Omega$ and a Lipschitz function $u\in C_0(\Omega)$ such that
\begin{equation}\label{alberti1}
\mu(A) \leq \epsilon \mu(\Omega),
\end{equation}
\begin{equation}\label{alberti2}
f_j = d^j u \text{ a.e. on } U_j \cap (\Omega \setminus A)
\end{equation}
for all $j\in J$,
\begin{equation}\label{alberti3}
\| \Lip_u \|_p \leq C \epsilon^{\frac{1}{p}-\frac{1}{\eta}}\left(\sum_{j\in J} (\LIP(\phi_j))^p \int_{\Omega \cap U_j}|f_j|^p \right)^{1/p}
\end{equation}
for all $p\in [1,\infty)$, and
\begin{equation}\label{alberti4}
\| \Lip_u \|_\infty \leq C \epsilon^{-\frac{1}{\eta}}\sup_{j\in J}\left(\LIP(\phi_j) \|f_j\|_\infty\right). 
\end{equation}

The constants $C,\eta>0$ depend only on the data of $X$.
\end{thm}

As $u$ is Lipschitz, the Poincar\'e inequality (see Lemma \ref{PIconsequences}) will allow us to also control the global Lipschitz constant of $u$ in Theorem \ref{alberti}, and conclude that
$$ \LIP(u) \leq C\epsilon^{-\frac{1}{\eta}}\sup_{j\in J} ((\LIP (\phi_j))\|f_j\|_\infty),$$
if the right-hand side is finite.

That the bounds (\ref{alberti3}) and (\ref{alberti4}) involve the the chart functions $\phi_j$ is in some sense inevitable, as one can easily discover by looking at the measurable differentiable structure $(\RR, \phi(x)=2x)$ on $\RR$. Note also that, unlike in the Euclidean setting of Theorem \ref{albertioriginal}, the notion of $C^1$ regularity is not defined in PI spaces. Thus, the natural regularity for our constructed function $u$ in Theorem \ref{alberti} is Lipschitz.

Our second result is the analog in PI spaces of Theorem \ref{mporiginal}:

\begin{thm}\label{mpthm}
Let $(X,d,\mu)$ be a PI space. Let $\{(U_j, \phi_j\colon U_j\rightarrow\RR^{k_j})\}$ be a measurable differentiable structure on $X$. Let $\Omega \subset X$ be open, let $\epsilon>0$, and let $\{f_j\colon U_j\cap\Omega \rightarrow \RR^{k_j}\}$ be a collection of Borel functions.

Then there is a continuous function $u$ on X that is differentiable almost everywhere and satisfies
\begin{equation}\label{mpthm1}
\|u\|_\infty \leq \epsilon \text{ and } \{u\neq 0\}\subset \Omega,
\end{equation}
\begin{equation}\label{mpthm2}
d^j u = f_j  \text{ a.e. in } U_j \cap \Omega
\end{equation}
for each $j\in J$, and
\begin{equation}\label{mpthm3}
\Lip_u = 0 \text{ everywhere in } X\setminus \Omega.
\end{equation}
\end{thm}

In particular, Theorems \ref{alberti} and \ref{mpthm} allow one to prescribe the horizontal derivatives of functions on the Heisenberg group, or to prescribe the one-dimensional Cheeger derivatives of functions on the Laakso-type spaces of \cite{La00} and \cite{CK13_PI}.

\section{Definitions and Preliminaries}\label{sectiondefs}
We will work with metric measure spaces $(X,d,\mu)$ such that $(X,d)$ is complete and $\mu$ is a Borel regular measure. If the metric and measure are understood, we will denote such a space simply by $X$. An open ball in $X$ with center $x$ and radius $r$ is denoted $B(x,r)$. If $B=B(x,r)$ is a ball in $X$ and $\lambda>0$, we write $\lambda B = B(x,\lambda r)$.

We generally use $C$ and $C'$ to denote positive constants that depend only on the quantitative data associated to the space $X$ (see below); their values may change throughout the paper.

If $\Omega\subset X$ is open, we let $C_c(\Omega)$ denote the space of continuous functions with compact support in $\Omega$. We also let $C_0(\Omega)$ denote the completion of $C_c(\Omega)$ in the supremum norm. Any function in $C_0(\Omega)$ admits a natural extension by zero to a continuous function on all of $X$.

Recall that a real-valued function $u$ on a metric space $(X,d)$ is \textit{Lipschitz} if there is a constant $L\geq 0$ such that
$$ |u(x) - u(y)| \leq L d(x,y) \text{ for all } x,y\in X. $$
The infimum of all $L\geq 0$ such that the above inequality holds is called the \textit{Lipschitz constant} of $u$ and is denoted $\LIP (u)$.

Given a real-valued (not necessarily Lipschitz) function $u$ on $X$, we also define its pointwise upper Lipschitz constant at points $x\in X$ by
$$ \Lip_u (x) = \limsup_{r\rightarrow 0} \frac{1}{r} \sup_{d(x,y) < r} |u(y) - u(x)|. $$

Two basic facts about $\Lip$ are easy to verify. First, for any two functions $f$ and $g$,
\begin{equation}\label{Lipsum}
\Lip_{f+g}(x) \leq \Lip_f(x) + \Lip_g(x).
\end{equation}
Second, if $f$ and $g$ are Lipschitz functions, then
\begin{equation}\label{Lipprod}
\Lip_{fg}(x) \leq f(x)(\Lip_g(x)) + g(x)(\Lip_f(x)).
\end{equation}

A non-trivial Borel regular measure $\mu$ on a metric space $(X,d)$ is a \textit{doubling measure} if there is a constant $C>0$ such that $\mu(B(x,2r) \leq C\mu(B(x,r))$ for every ball $B(x,r)$ in $X$. The existence of a doubling measure $\mu$ on $(X,d)$ implies that $(X,d)$ is a \textit{doubling metric space}, i.e. that every ball can be covered by at most $N$ balls of half the radius, for some fixed constant $N$. In particular, a complete metric space with a doubling measure is \textit{proper}: every closed, bounded subset is compact. 

\begin{definition}
A metric measure space $(X,d,\mu)$ is a \textit{PI space} if $(X,d)$ is complete, $\mu$ is a doubling measure on $X$ and $(X,d,\mu)$ satisfies a ``$(1,q)$-Poincar\'e inequality'' for some $1\leq q < \infty$: There is a constant $C>0$ such that, for every compactly supported Lipschitz function $f \colon X\rightarrow\RR$ and every open ball $B$ in $X$,
$$\fint_B |f-f_B|d\mu \leq C(\diam B)\left(\fint_{CB} (\Lip_f)^q d\mu \right)^{1/q}.$$
(Here the notations $\fint_E g d\mu$ and $g_E$ both denote the average value of the function $g$ on the set $E$, i.e., $\frac{1}{\mu(E)}\int_E g d\mu$.)
\end{definition}
This definition can be found in \cite{Ke03_PI}; it is equivalent to other versions of the Poincar\'e inequality in metric measure spaces, such as the original one of \cite{HK98}. If $X$ is a PI space, then the collection of constants associated to the doubling property and Poincar\'e inequality on $X$ are known as the \textit{data of $X$}.

In addition to providing a differentiable structure (see below), the PI space property of $X$ will supply two other key facts for us, summarized in the following proposition.
\begin{prop}\label{PIconsequences}
Let $(X,d,\mu)$ be a PI space. Then there is a constant $C>0$, depending only on the data of $X$, such that the following two statements hold:
\begin{enumerate}[\normalfont (a)]
\item $X$ is \textit{quasiconvex}, meaning that any two points $x,y \in X$ can be joined by a rectifiable path of length at most $Cd(x,y)$.
\item For any bounded Lipschitz function $u$ on $X$, 
$$ \LIP (u) \leq C \| \Lip_u\|_\infty. $$
\end{enumerate}
\end{prop}
\begin{proof}
The first statement can be found in Theorem 17.1 of \cite{Ch99}. The second can be found (in greater generality than we need here) in \cite{DJS12}, Theorem 4.7.
\end{proof}

The following definition is essentially due to Cheeger, in Section 4  of \cite{Ch99}. The form we state can be found in Definition 2.1.1 of \cite{Ke03} (see also \cite{BS13}). The notation $\langle \cdot, \cdot \rangle$ denotes the standard inner product on Euclidean space of the appropriate dimension.
\begin{definition}\label{diffstructure}
Let $(X,d,\mu)$ be a metric measure space. Let $\{U_j\}_{j\in J}$ be a collection of pairwise disjoint measurable sets covering $X$, let $\{k_j\}_{j\in J}$ be a collection of non-negative integers, and let $\{\phi_j\colon X\rightarrow \RR^{k_j}\}_{j\in J}$ be a collection of Lipschitz functions.

We say that the collection $\{(U_j, \phi_j)\}$ forms a \textit{measurable differentiable structure} for $X$ if the following holds: For every Lipschitz function $u$ on $X$ and every $j\in J$, there is a Borel measurable function $d^j u \colon  U_j \rightarrow \RR^{k_j}$ such that, for almost every $x\in U_j$,
\begin{equation}\label{differentiable}
\lim_{y\rightarrow x} \frac{|u(y)-u(x) - \langle d^ju(x) , (\phi_j(y) - \phi_j(x)) \rangle|}{d(y,x)} = 0.
\end{equation}
Furthermore, the function $d^j u$ should be unique (up to sets of measure zero).
\end{definition}
We call each pair $(U_j, \phi_j)$ a \textit{chart} for the differentiable structure on $X$. For more background on differentiable structures (also called ``strong measurable differentiable structures'' and ``Lipschitz differentiability spaces'') see \cite{Ke03, Ba13}.

Note that the defining property (\ref{differentiable}) for a measurable differentiable structure can be more succinctly rephrased as
$$ \Lip_{u - \langle d^ju(x) , \phi_j \rangle}(x) = 0.$$

The link between PI spaces and measurable differentiable structures is given by the following theorem of Cheeger, one of the main results of \cite{Ch99}. (See also \cite{Ke03, KM11, Ba13} for alternate approaches.)

\begin{thm}[\cite{Ch99}, Theorem 4.38]\label{Cheegertheorem}
Every PI space $X$ supports a measurable differentiable structure
$$\{(U_j, \phi_j\colon X\rightarrow\RR^{k_j})\},$$
and the dimensions $k_j$ of the charts $U_j$ are bounded by a uniform constant depending only on the constants associated to the doubling property and Poincar\'e inequality of $X$.
\end{thm}

If $X$ supports a measurable differentiable structure, then it generally supports many other equivalent ones. For example, the sets $U_j$ may be decomposed into measurable pieces or the functions $\phi_j$ rescaled without altering the properties in Definition \ref{diffstructure}. At times, it will be helpful to assume certain extra properties of the charts. 
\begin{definition}\label{normalized}
A measurable differentiable structure $\{(U_j, \phi_j\colon X\rightarrow\RR^{k_j})\}_{j\in J}$ is \textit{normalized} if, for each $j\in J$, there exists $c_j>0$, such that
\begin{equation}\label{normalized1}
U_j \text{ is closed,}
\end{equation}
\begin{equation}\label{normalized2}
\LIP(\phi_j)=1 \text{, and}
\end{equation}
\begin{equation}\label{normalized3}
|d^j u(x)| \leq c_j \Lip_u(x) \text{ whenever } u \text{ is differentiable at } x\in U_j.
\end{equation}
\end{definition}

The definition of a normalized chart is a minor modification of the notion of a ``structured chart'', due to Bate (\cite{Ba13}, Definition 3.6). The following lemma, essentially due to Bate, says that a given chart structure on $X$ can always be normalized by rescaling and chopping.
\begin{lemma}\label{batelemma}
Let $X$ be a PI space and let $\{(U_j, \phi_j\colon X\rightarrow\RR^{k_j})\}_{j\in J}$ be a measurable differentiable structure on $X$. Then there exists a collection of sets $\{U_{j,k}\}_{j\in J, k\in K_j}$ such that
\begin{itemize}
\item each set $U_{j,k}$ is contained in $U_j$ and
\item $\{(U_{j,k}, (\LIP(\phi_j))^{-1} \phi_j)\}_{j\in J, k\in K_j}$ is a normalized measurable differentiable structure on $X$.
\end{itemize}
\end{lemma}
\begin{proof}
By Lemma 3.4 of \cite{Ba13}, we can decompose each chart $U_j$ into charts $U_{j,k}$ such that the measurable differentiable structure $\{(U_{j,k}, \phi_j)\}$ satisfies (\ref{normalized3}). 

As $(\LIP(\phi_j))^{-1} \phi_j)$ is just a rescaling of $\phi_j$, the chart $(U_{j,k}, (\LIP(\phi_j))^{-1} \phi_j)$ still possesses property (\ref{normalized3}) (with a different constant $c_{j,k}$).

As a final step, we decompose each $U_{j,k}$ into closed sets, up to measure zero, while maintaining the same chart functions.
\end{proof}

For technical reasons, it will be convenient in the proofs of Theorems \ref{alberti} and \ref{mpthm} that the measurable differentiable structure is normalized. That this can be done without loss of generality is the content of the following simple lemma.
\begin{lemma}\label{WLOGlemma}
To prove Theorems \ref{alberti} and \ref{mpthm}, we can assume without loss of generality that the measurable differentiable structure $\{(U_j,\phi_j)\}$ is normalized.
\end{lemma}
\begin{proof}
Assume that we can prove Theorems \ref{alberti} and \ref{mpthm} if the charts involved are normalized.

Suppose $(U_j, \phi_j)$ is an arbitrary (not necessarily normalized) measurable differentiable structure on $X$. Let $\Omega \subset X$ be open with $\mu(\Omega)<\infty$, let $\{f_j\colon U_j \cap \Omega \rightarrow \RR^{k_j}\}_{j\in J}$ be a collection of Borel functions, and let $\epsilon>0$.

By Lemma \ref{batelemma}, there is a normalized measurable differentiable structure 
\begin{equation}\label{newstructure}
\{(U_{j,k}, (\LIP(\phi_j))^{-1} \phi_j)\}_{j\in J, k\in K_j}
\end{equation}
on $X$, where each $U_{j,k}$ is contained in $U_j$.

Let $g_{j,k} = (\LIP \phi_j) f_j$. Apply Theorem \ref{alberti} to the normalized measurable differentiable structure (\ref{newstructure}), with the functions $g_{j,k}$ and the same parameter $\epsilon$. We immediately obtain an open set $A\subset \Omega$ and a Lipschitz function $u\in C_0(\Omega)$ that satisfy all four requirements of Theorem \ref{alberti}.

A similar argument applies to reduce Theorem \ref{mpthm} to the normalized case.
\end{proof}

The original arguments of \cite{Al91} and \cite{MP08} to prove Theorems \ref{albertioriginal} and \ref{mporiginal} use the dyadic cube decomposition of Euclidean space. We will use the analogous decomposition in arbitrary doubling metric spaces provided by a result of Christ \cite{Ch90}.

\begin{prop}[\cite{Ch90}, Theorem 11]\label{cubes}
Let $(X,d,\mu)$ be a doubling metric measure space. Then there exist constants $c\in (0,1)$, $\eta>0$, $a_0>0$, $a_1>0$, and $C_1>0$ such that for each $k\in \mathbb{Z}$ there is a collection $\Delta_k = \{Q^k_i : i \in I_k\}$ of disjoint open subsets of $X$ with the following properties:
\begin{enumerate}[\normalfont (i)]
\item \label{cubesfull} For each $k\in \mathbb{Z}$, $\mu(X \setminus \cup_i Q^k_i) = 0$.
\item \label{cubesball} For each $k\in \mathbb{Z}$ and $i\in I_k$, there is a point $z^k_i\in Q^k_i$ such that
$$ B(z^k_i, a_0 c^k) \subset Q^k_i \subset B(z^k_i, a_1 c^k). $$
\item \label{cubesboundary} For each $k\in \mathbb{Z}$ and $i\in I_k$, and for each $t>0$,
$$ \mu\left( \{ x\in Q^k_i : \dist(x, X\setminus Q^k_i) \leq tc^k \} \right) \leq C_1 t^\eta \mu(Q^k_i). $$
\item \label{cubesdisjoint} If $\ell \geq k$, $Q\in \Delta_\ell$, and $Q'\in \Delta_k$, then either $Q\subset Q'$ or $Q\cap Q'=\emptyset$.
\item \label{cubesnested} For each $k\in \mathbb{Z}$, each $i\in I_k$, and each integer $\ell< k$ there is a unique $j\in I_\ell$ such that $Q^k_i \subset Q^\ell_j$.
\end{enumerate}
\end{prop}
We refer to the elements of any $\Delta_k$ as cubes.

The next lemma is one of the primary differences between our proof of Theorem \ref{alberti} and the proof of Theorem \ref{albertioriginal} from \cite{Al91}. It allows us to replace a single-scale argument in \cite{Al91} by an argument that uses multiple scales simultaneously, which will allow us to deal with the presence of multiple charts.
\begin{lemma}\label{multiscale}
Let $(X,d,\mu)$ be a complete doubling metric measure space. Suppose that $\mu(X \setminus \cup_{j\in J} U_j)=0$, where the sets $U_j$ are disjoint and measurable. Fix $\gamma>0$ and positive numbers $\{\delta_j\}_{j\in J}$. Then we can find a collection $\mathcal{T}$ of pairwise disjoint cubes in $X$ (of possibly different scales) such that the following conditions hold:
\begin{enumerate}[(i)]
\item \label{multiscale1} $\mu(X\setminus \bigcup_{T\in \mathcal{T}} T) = 0$.
\item \label{multiscale2} There is a map $j\colon \mathcal{T}\rightarrow J$ such that
\begin{equation}\label{good1}
\mu(U_{j(T)} \cap T) \geq (1-\gamma)\mu(T)
\end{equation}
and
\begin{equation}\label{good2}
 \diam T < \delta_{j(T)}
\end{equation}
for each $T\in \mathcal{T}$.
\end{enumerate}
\end{lemma}
\begin{proof}
Let us call a cube $T\in \Delta_k$ ``good for $j$'' if it satisfies (\ref{good1}) and (\ref{good2}) with $j(T)=j$, and let us call $T$ ``good'' if it is good for some $j\in J$. Finally, let us call $T$ ``bad'' if it is not good. Write $\Delta^g_k$ for the sub-collection of $\Delta_k$ consisting of good cubes.

We then define our collection of cubes $\mathcal{T}$ to be
$$ \mathcal{T} = \bigcup_{k=1}^\infty \{ T\in \Delta^g_k : \text{for every $1\leq k'< k$ and every $Q\in \Delta_{k'}$ containing $T$, $Q$ is bad}\}.$$
In other words, our collection consists of all cubes that are the first good cube among all their ancestors of scales below $1$. Note that any two distinct cubes in $\mathcal{T}$ are disjoint: if not, then one would contain the other, forcing the larger one to be bad.

For each cube $T$ in this collection $\mathcal{T}$, define $j(T)$ to be a choice of $j\in J$ such that $T$ is good for $j$. The collection $\mathcal{T}$ and the map $j:\mathcal{T}\rightarrow J$ then automatically satisfy condition (\ref{multiscale2}) of the Lemma.

To verify (\ref{multiscale1}), we show that almost every point $x\in X$ is contained in one of the cubes $T\in \mathcal{T}$. Let 
$$ Z = X\setminus \bigcap_{k\in\mathbb{Z}} \bigcup_{\ell \in I_k} Q^k_\ell, $$
so that $\mu(Z)=0$ by Proposition \ref{cubes} (\ref{cubesfull}).

Let $x\in X\setminus Z$ be a point of $\mu$-density of some $U_{j_0}$. We claim that $x\in T$ for some $T\in \mathcal{T}$. Suppose, to the contrary, that $x\notin T$ for any $T\in \mathcal{T}$. Then $x$ lies in an infinite nested sequence of bad cubes. But this is impossible: if an infinite nested sequence of cubes satisfied $Q_1 \supset Q_2 \supset \dots \ni x$, then eventually some $Q_i$ would be good for $j_0$, and the first such good cube would be in $\mathcal{T}$.

So $\cup_{T\in \mathcal{T}} T$ contains almost every point in $\left(\cup_{j\in J} U_j \right)\cap (X\setminus Z)$, which is almost every point of $X$.
\end{proof}

The following lemma will ensure that we obtain a Lipschitz function in our construction. Recall the definition of quasiconvexity from Proposition \ref{PIconsequences}.

\begin{lemma}\label{lipschitzlemma}
Let $X$ be a complete and quasiconvex metric space and let $u\colon X\rightarrow \RR$ be a function on $X$. Suppose that there are pairwise disjoint open sets $A_i\subset X$ ($i\in I$), and a constant $L\geq 0$ such that
\begin{equation}\label{liplemma1}
\LIP(u|_{A_i}) \leq L \text{ for each } i\in I
\end{equation}
and
\begin{equation}\label{liplemma2}
u=0 \text{ on } B = X\setminus \cup_{i\in I} A_i.
\end{equation}
Then $u$ is $2CL$-Lipschitz on $X$, where $C$ is the quasiconvexity constant of $X$.
\end{lemma}
\begin{proof}
Without loss of generality, we may assume that $A_i\neq X$ for each $i\in I$, otherwise the lemma is trivial. Note also that the assumption (\ref{liplemma1}) improves immediately to
$$ \LIP(u|_{\overline{A_i}}) \leq L \text{ for each } i\in I.$$

Fix points $x,y\in X$. We will show that
\begin{equation}\label{liplemmabound}
|u(x)-u(y)| \leq 2CLd(x,y),
\end{equation}
where $C$ is the quasiconvexity constant of $X$.

Using the quasiconvexity of $X$, choose a rectifiable path $\gamma\colon [0,1]\rightarrow X$ of length at most $Cd(x,y)$ such that $\gamma(0)=x$ and $\gamma(1)=y$.

Case 1: Suppose that, for some $i,j \in I$, we have $x\in A_i$ and $y\in A_j$.
In this case, we may also suppose that $i\neq j$, otherwise (\ref{liplemmabound}) follows from the assumption (\ref{liplemma1}). Let $t_0 = \inf\{t: \gamma(t)\notin A_i\}$ and $t_1 = \sup\{t: \gamma(t)\notin A_j\}$. By basic topology, $ \gamma(t_0) \in \partial A_i \subset B$ and $\gamma(t_1) \in \partial A_j \subset B. $
Thus, we have
\begin{align*}
|u(x) - u(y)| &\leq |u(x) - u(\gamma(t_0))| + |u(\gamma(t_0)) - u(\gamma(t_1))| + |u(\gamma(t_1)) - u(y)|\\
 &\leq Ld(x,\gamma(t_0)) + 0 + Ld(y, \gamma(t_1))\\
&\leq (2L) \text{length}(\gamma)\\
&\leq 2CLd(x,y)
\end{align*}

Case 2: Suppose that $x\in A_i$ for some $i\in I$ and that $y\in B$ (or vice versa).
We then have that $u(y) = 0$ and
$$ d(x,y) \geq \frac{1}{C} \text{length}(\gamma) \geq \frac{1}{C} \dist(x,\partial A_i).$$
Thus,
$$ |u(x) - u(y)| = |u(x)| \leq L\dist(x, \partial A_i) \leq CL d(x,y). $$

Case 3: Suppose that $x\in B$ and $y\in B$.
Then $u(x) = u(y) = 0$.
\end{proof}
Note that Lemma \ref{lipschitzlemma} is false without assumption (\ref{liplemma2}), as the Cantor staircase function shows.

The following lemma is due to Francos (\cite{Fr12}, Lemma 2.3). Although Francos stated it only for subsets of $\RR^n$, the proof works equally well in our setting.
\begin{lemma}\label{francos}
Let $X$ be a PI space and let $f$ be a Borel function from an open set $\Omega\subset X$, with $\mu(\Omega)<\infty$, into some $\RR^N$. Then, for any $\epsilon>0$, there is a compact set $K\subset \Omega$ and a continuous function $g$ on $\Omega$ such that
\begin{enumerate}[\normalfont (a)]
\item $\mu(U\setminus K)<\epsilon$,
\item $f=g$ on $K$, 
\item $\int_\Omega |g|^p \leq 2\int_\Omega |f|^p$ for all $p\in [1,\infty)$, and
\item $\|g\|_\infty \leq 2 \|f\|_\infty$.
\end{enumerate}
\end{lemma}

\section{Proof of Theorem \ref{alberti}}\label{sectionproof1}
Our main lemma is the analog of Lemma 7 of \cite{Al91}:

\begin{lemma}\label{mainlemma}
Let $(X,d,\mu)$ be a PI space and let $\{(U_j, \phi_j\colon X\rightarrow\RR^{k_j})\}_{j\in J}$ be a normalized measurable differentiable structure on $X$. Suppose that $\Omega \subset X$ is open with $\mu(\Omega)<\infty$ and $\Omega \neq X$, and that $\{f_j\colon \Omega \rightarrow \RR^{k_j}\}$ is a uniformly bounded collection of continuous functions, i.e., that $\sup_{j\in J} \|f_j\|_\infty < \infty$. Fix $\alpha, \epsilon>0$.

Then there exists a compact set $K\subset \Omega$ and a Lipschitz function $u\in C_c(\Omega)$ such that the following conditions hold: 
\begin{equation}\label{lemma1}
\mu(\Omega \setminus K) \leq \epsilon \mu(\Omega).
\end{equation}
\begin{equation}\label{lemma2}
|f_j - d^j u|\leq \alpha \text{ a.e. on } U_j \cap K.
\end{equation}
\begin{equation}\label{lemma3}
\| \Lip_u \|_p \leq C' \epsilon^{\frac{1}{p}-\frac{1}{\eta}}\left(\sum_{j\in J} \int_{U_j\cap \Omega}|f_j|^p \right)^{1/p} \text{ for all }p\in [1,\infty).
\end{equation}
\begin{equation}\label{lemma4}
\|\Lip_u\|_\infty \leq C' \epsilon^{-\frac{1}{\eta}}\sup_{j\in J} \|f_j\|_\infty.
\end{equation}

The constants $\eta, C'>0$ depend only on the data of $X$.
\end{lemma}
\begin{proof}
Without loss of generality, we assume that $\epsilon<1$.

Fix a compact set $K'\subset \Omega$ such that $\mu(\Omega\setminus K')< \frac{\epsilon}{4} \mu(\Omega)$. For each $j\in J$, choose $\delta_j>0$ small enough such that
\begin{equation}\label{deltaj1}
 \text{ if } |x-y|< \delta_j \text{ and } x\in K' \text{, then } |f_j(x) - f_j(y)| < \alpha/2,
\end{equation}
and
\begin{equation}\label{deltaj2}
\delta_j < \dist(K', X\setminus \Omega).
\end{equation}

Using Lemma \ref{multiscale}, we find a collection $\mathcal{T}$ of pairwise disjoint cubes covering almost all of $X$, and a map $j\colon \mathcal{T} \rightarrow J$ such that
\begin{equation}\label{lemmacubes1}
 \mu(U_{j(T)} \cap T)  \geq \left(1-\frac{\epsilon}{4}\right)\mu(T),
\end{equation}
and
\begin{equation}\label{lemmacubes2}
\diam (T) < \delta_{j(T)}.
\end{equation}
for each $T\in \mathcal{T}$.

Consider the sub-collection consisting of all cubes $T\in \mathcal{T}$ such that $T\cap K'\neq \emptyset$. Index these cubes $\{T_i\}_{i\in I}$, and write $j(i)$ for $j(T_i)$. By (\ref{deltaj2}) and (\ref{lemmacubes2}), each cube $T_i$ ($i\in I$) lies in $\Omega$. 

For each $i\in I$, define $S_i \subset T_i$ as
$$ S_i = \{ x\in T_i : \dist(x, X\setminus T_i) \geq tc^k \}, $$
where $k$ is such that $T\in\Delta_k$ and 
$$t = (\epsilon/4C_1)^{1/\eta}$$
is fixed. This value of $t$ was chosen to ensure (by Proposition \ref{cubes} (\ref{cubesboundary})) that
$$ \mu(T_i \setminus S_i) \leq C_1 t^\eta \mu(T_i) = \frac{\epsilon}{4} \mu(T_i). $$
Note that $S_i$ is a compact subset of the open set $T_i$. Let $z_i$ be a ``center'' of $T_i$ as in Proposition \ref{cubes} (\ref{cubesball}), so that $T_i$ both contains and is contained in a ball centered at $z_i$ of radius approximately $\diam T_i$.

For each cube $T_i$ in our collection, define $a_i\in\RR^{k_{j(i)}}$ by  
$$ a_i = \mu(U_{j(i)} \cap T_i)^{-1} \int_{U_{j(i)} \cap T_i} f_{j(i)}. $$
Note that the collection $\{|a_i|\}_{i\in I}$ is bounded, because the collection $\{f_j\}_{j\in J}$ is uniformly bounded.

Let $\psi_i\colon X\rightarrow \RR_+$ be a Lipschitz function such that $\psi_i = 1$ on $S_i$, $\psi_i = 0$ off $T_i$, and $\LIP(\psi_i) \leq C(tc^k)^{-1} \leq C (\diam T_i)^{-1} \epsilon^{-1/\eta}$. (Here $C$ is some constant depending only on the data of $X$.) By slightly widening the regions where $\psi$ is constant, we can also easily arrange that $\Lip_{\psi_i} = 0$ everywhere in $S_i$ and in $X\setminus T_i$. 

Define $u\colon  X \rightarrow \RR$ by
\begin{equation}\label{udefinition}
u(x) = \sum_{i\in I} \psi_i(x) \langle a_i, \phi_{j(i)}(x)-\phi_{j(i)}(z_i) \rangle.
\end{equation}
A simple calculation shows that, for each $i\in I$,
$$ \LIP(u|_{T_i}) \leq C\epsilon^{-1/\eta}|a_i| \leq C\epsilon^{-1/\eta}\sup_i|a_i| < \infty.$$
(Here we used the assumption that the measurable differentiable structure is normalized, and therefore $\LIP(\phi_j)\leq 1$ for each $j\in J$.)

Thus, as $u=0$ outside $\bigcup_{i\in I} T_i$, we see that $u$ is Lipschitz on $X$ by Lemma \ref{lipschitzlemma}. In addition, $u\in C_c(\Omega)$, with $\text{supp } u \subset \overline{\bigcup_{i \in I} T_i} \subset \Omega$, and $d^{j(i)} u = a_i$ a.e. on $S_i\cap U_{j(i)}$.

Let $K_1 = \cup_{i\in I} (S_i \cap U_{j(i)})$, and let $K$ be a compact subset of $K_1$ such that
$$ \mu( K_1 \setminus K) \leq \frac{\epsilon}{4}\mu(\Omega).$$
To verify (\ref{lemma1}), note that
$$ \mu(T_i\setminus (S_i\cap U_{j(i)})) \leq \mu(T_i\setminus S_i) + \mu(T_i\setminus U_{j(i)}) \leq \frac{\epsilon}{4} \mu(T_i) + \frac{\epsilon}{4}\mu(T_i) \leq \frac{\epsilon}{2}\mu(T_i)$$
for each $i\in I$. Therefore,
$$ \mu(\Omega\setminus K_1) \leq \mu(\Omega\setminus K') + \sum_i \mu(T_i\setminus (S_i\cap U_{j(i)})) \leq \frac{3\epsilon}{4} \mu(\Omega),$$
and so
$$\mu(\Omega \setminus K) \leq \mu(\Omega\setminus K_1) + \mu(K_1\setminus K) \leq \epsilon \mu(\Omega).$$

Let us now verify (\ref{lemma2}). Suppose that $x\in U_j \cap K$ for some $j\in J$. Then $x\in S_i \cap U_j$ for some $i\in I$ such that $j(i)=j$. Therefore, by (\ref{deltaj1}), $|f_j(x) - a_i|<\alpha$. So $|f_j - a_i| < \alpha$ on $S_i \cap U_j$. Now, since $d^j u = a_i$ almost everywhere in $U_j \cap S_i$, we see that
 $$|f_j - d^ju|<\alpha$$
almost everywhere in $U_j \cap S_i$. This verifies (\ref{lemma2}), as $K\subset \cup_{i} (S_i \cap U_{j(i)})$.

Finally, we must check (\ref{lemma3}) and (\ref{lemma4}). Observe that if $T$ is a cube in $\mathcal{T}$ and $x\in T$, then the sum (\ref{udefinition}) defining $u$ consists of at most one non-zero term. Therefore, for such $x$, we have by (\ref{Lipprod}) that
\begin{align}\label{Lipubound}
\Lip_u(x) &\leq \sup_{i\in I}  \left( \Lip_{\psi_i}(x) |\langle a_i, \phi_{j(i)}(x)-\phi_{j(i)}(z_i) \rangle| + |a_i| \psi_i(x)\right) \nonumber \\
&\leq \sum_{i\in I} \Lip_{\psi_i}(x) |\langle a_i, \phi_{j(i)}(x)-\phi_{j(i)}(z_i) \rangle| + \sum_{i\in I} |a_i| \psi_i(x).
\end{align}
Because almost every $x\in X$ is contained in some $T\in\mathcal{T}$, we have the bound (\ref{Lipubound}) for almost every $x\in \Omega$.

Recalling our normalization that $\LIP(\phi_j)\leq 1$ for all $j\in J$, we see from (\ref{Lipubound}) that, for all $1\leq p < \infty$,
\begin{align*}
\| \Lip_u \|_p &\leq \left( \sum_{i\in I} \left(\|\Lip_{\psi_i}\|_\infty |a_i| (\diam T_i) \right)^p\mu(T_i\setminus S_i) \right)^{1/p} + \left(\sum_{i\in I} |a_i|^p \mu(T_i)\right)^{1/p}\\
&\leq \left( \sum_{i\in I} \left(C\epsilon^{-1/\eta} |a_i| \right)^p \epsilon\mu(T_i) \right)^{1/p} + \left(\sum_{i\in I} |a_i|^p \mu(T_i)\right)^{1/p}\\
&\leq \left( C \epsilon^{\frac{1}{p}-\frac{1}{\eta}}  + 1\right)\left(\sum_{i\in I} |a_i|^p \mu(T_i)\right)^{1/p}\\
&\leq \left( C \epsilon^{\frac{1}{p}-\frac{1}{\eta}}  + 1\right)\left(\sum_{i\in I} \frac{\mu(T_i)}{\mu(T_i\cap U_{j(i)})} \int_{T_i\cap U_{j(i)}} |f_{j(i)}|^p\right)^{1/p}\\
&\leq 2\left( C \epsilon^{\frac{1}{p}-\frac{1}{\eta}} + 1\right)\left(\sum_{j\in J} \int_{\Omega\cap U_j}|f_j|^p\right)^{1/p}
\end{align*}
Note that in the last inequality we used (\ref{lemmacubes1}).

The case $p=\infty$, namely (\ref{lemma4}), follows from this by a limiting argument, or can alternatively be derived the same way. This completes the proof of Lemma \ref{mainlemma}.
\end{proof}

\begin{proof}[Proof of Theorem \ref{alberti}]
By Lemma \ref{WLOGlemma}, we may assume that the measurable differentiable structure is normalized.

It will also be convenient to assume that $\Omega$ is a proper subset of $X$, i.e., that $\Omega\neq X$. We may assume this without loss of generality: If $\Omega= X$, we replace $\Omega$ by $\Omega' = X\setminus\{x_0\}$ for some arbitrary $x_0\in X$. Proving Theorem \ref{alberti} for $\Omega'$ also proves it for $\Omega$.

Finally, we may also assume that $\epsilon<1$ and that $\sup_{j\in J} \|f_j\|_\infty>0$. We also extend each $f_j$ from $U_j \cap \Omega$ to all of $\Omega$ by setting $f_j = 0$ off of $U_j$. The proof now proceeds in two steps.

\textbf{Step 1:} Assume that the functions $f_j$ are uniformly bounded, i.e., that $\sup_{j\in J} \| f_j\|_\infty < \infty$. 

Let $\{\alpha_n\}_{n\geq 1}$ be a decreasing sequence of positive real numbers with $\alpha_1 \leq \sup_{j\in J} \| f_j\|_\infty$, to be chosen later. For each integer $n\geq 0$, we will inductively build:
$$ \text{ a Lipschitz function } u_n \in C_c(\Omega,\RR), $$
$$ \text{ a compact set } K_n \subset \Omega \text{, and}$$
$$ \text{ a collection of continuous functions } \{f_j^n\colon \Omega \rightarrow \RR^{k_j}\}_{j\in J}.$$
Let $u_0 = 0$. For each $j\in J$, we apply Lemma \ref{francos}, to find a compact set $K_{0,j} \subset \Omega$ with $\mu(\Omega\setminus K_{0,j})<2^{-1}2^{-j}\epsilon\mu(\Omega)$, and a continuous function $f^0_j$ on $\Omega$ such that $f^0_j = f_j$ on $K_{0,j}$,
\begin{equation}\label{f0p}
 \int_\Omega |f^0_j|^p \leq 2\int_\Omega |f_j|^p = 2\int_{\Omega\cap U_j} |f_j|^p
\end{equation}
for all $1\leq p <\infty$, and
\begin{equation}\label{f0infty}
 \sup |f^0_j| \leq 2 \|f_j\|_\infty.
\end{equation}
Let $K_0 = \cap_{j\in J} K_{0,j}$. This completes stage $n=0$ of the construction.

Suppose now that $u_{n-1}, K_{n-1}, \{f^{n-1}_j\}_{j\in J}$ have been constructed. Apply Lemma \ref{mainlemma} to get a compact set $\tilde{K}_n\subset \Omega$ and a Lipschitz function $u_n\in C_c(\Omega)$ such that
$$\mu(\Omega \setminus \tilde{K}_n) \leq 2^{-(n+2)}\epsilon \mu(\Omega)$$,
$$|f^{n-1}_j(x) - (d^j u_n)(x)| \leq \alpha_n/2$$
for every $j\in J$ and almost every $x\in U_j \cap \tilde{K}_n$,
$$\|\Lip_{u_n}\|_p \leq C'(2^{-(n+2)}\epsilon)^{\frac{1}{p}-\frac{1}{\eta}} \left(\sum_j \int_{U_j \cap \Omega}|f_j^{n-1}|^p\right)^{1/p}$$
for every $p\in [1,\infty)$, and
$$ \|\Lip_{u_n}\|_\infty \leq C'(2^{-(n+2)}\epsilon)^{-\frac{1}{\eta}} \sup_{j\in J} \|f_j^{n-1}\|_\infty.$$

Given $j\in J$, define $\tilde{f}^n_j(x) \colon \Omega \rightarrow \RR^{k_j}$ by
\begin{equation}\label{ftildedef}
\tilde{f}^n_j(x) = 
\begin{cases}
f^{n-1}_j(x) - d^j u_n(x) & \text{ if } x\in U_j \cap \tilde{K}_n\\
0 & \text{otherwise}.
\end{cases}
\end{equation}

For each $j\in J$, apply Lemma \ref{francos} to find a compact set $K_{n,j}\subset \Omega$ and a continuous map $f^n_j\colon \Omega\rightarrow\RR^{k_j}$ such that 
\begin{equation}\label{fnjdef1}
 \mu(\Omega\setminus K_{n,j})\leq 2^{-(n+2)}2^{-j}\epsilon \mu(\Omega),
\end{equation}
\begin{equation}\label{fnjdef2}
 f^n_j = \tilde{f}^n_j \text{ on } K_{n,j}\text{, and}
\end{equation}
\begin{equation}\label{fnjdef3}
 \|f^n_j\|_\infty \leq 2\|\tilde{f}^n_j\|_\infty \leq \alpha_n.
\end{equation}

Let $K_n = \tilde{K}_n \cap \left(\bigcap_{j\in J} K_{n,j}\right)$, so that
$$\mu(\Omega \setminus K_n) \leq \mu(\tilde{K}_n) + \sum_{j\in J} \mu(\Omega\setminus K_{n,j}) \leq 2^{-(n+1)}\epsilon \mu(\Omega).$$

This completes stage $n$ of the inductive construction.

Now let
$$A=\Omega \setminus \cap_{n=0}^\infty K_n = \cup_{n=0}^\infty (\Omega\setminus K_n)$$
and
$$ u = \sum_{n=0}^\infty u_n.$$

Note that
$$ \mu(A) \leq \sum_{n=0}^\infty \mu(\Omega \setminus K_n) \leq \epsilon \mu(\Omega),$$
so (\ref{alberti1}) holds.

Purely for notational convenience, we now define a real-valued function
$$ F = \sum_{j\in J} \chi_{U_j} |f_j|, $$
so that
$$ \|F\|_p = \left(\sum_j \int_{U_j\cap\Omega} |f_j|^p\right)^{1/p} $$
for every $p\in [1,\infty)$ and
$$ \|F\|_\infty = \sup_{j\in J} \|f_j\|_\infty. $$
Note that, if $p\in [1,\infty)$, we have, by (\ref{f0p}) and (\ref{f0infty}), that
$$ \left(\sum_{j\in J} \int_{U_j\cap\Omega} |f^0_j|^p\right)^{1/p} \leq 2 \|F\|_p \hspace{0.2in} \text{ and } \hspace{0.2in}  \sup_{j\in J} \|f^0_j\|_\infty \leq 2\|F\|_\infty$$
In addition, $\|F\|_p$ is non-zero and finite for every $p\in[1,\infty]$, by our assumption that $0< \sup_{j\in J} \|f_j\|_\infty < \infty$.

We now calculate that, for $p\in [1,\infty)$,
\begin{align}
\sum_{n=1}^\infty \|\Lip_{u_n}\|_p &\leq \sum_{n=1}^\infty C'2^{\frac{n+2}{\eta}}\epsilon^{\frac{1}{p}-\frac{1}{\eta}}\left(\sum_{j\in J} \int_{U_j\cap \Omega}|f^{n-1}_j|^p\right)^{1/p} \nonumber \\
&\leq 2C'\epsilon^{\frac{1}{p}-\frac{1}{\eta}}\left(\|F\|_p + \sum_{n=1}^\infty 2^{\frac{n+2}{\eta}}\left(\sum_{j\in J} \int_{U_j\cap \Omega}|f^{n}_j|^p\right)^{1/p}\right) \nonumber\\
&\leq 2C'\epsilon^{\frac{1}{p}-\frac{1}{\eta}}\left(\|F\|_p + \sum_{n=1}^\infty 2^{\frac{n+2}{\eta}} \left(\sum_{j\in J} \|f^n_j \|_\infty ^p \mu(U_j\cap\Omega) \right)^{1/p}   \right) \nonumber\\
&\leq 2C'\epsilon^{\frac{1}{p}-\frac{1}{\eta}}\left(\|F\|_p + \sum_{n=1}^\infty 2^{\frac{n+2}{\eta}} \alpha_n \mu(\Omega)^{1/p} \right) \nonumber\\
&\leq 2C'\epsilon^{\frac{1}{p}-\frac{1}{\eta}} \|F\|_p \left(1+\frac{\mu(\Omega)^{1/p}}{\|F\|_p} \sum_{n=1}^\infty 2^{\frac{n+2}{\eta}} \alpha_n\right). \label{bigsum}
\end{align}
A similar calculation shows that (\ref{bigsum}) also holds if $p=\infty$.

The function $p\mapsto \frac{\mu(\Omega)^{1/p}}{\|F\|_p}$ is continuous for $p\in [1,\infty]$, and therefore has an upper bound $M>0$. Choose our sequence $\{\alpha_n\}$ to satisfy
$$\sum 2^{\frac{n+2}{\eta}} \alpha_n \leq 1/M.$$
Then the calculation (\ref{bigsum}) yields that
\begin{equation}\label{lipunsum}
\sum_{n=1}^\infty \|\Lip_{u_n}\|_p \leq 4C'\epsilon^{\frac{1}{p}-\frac{1}{\eta}}\|F\|_p < \infty
\end{equation}
for any $p\in [1,\infty]$.

Proposition \ref{PIconsequences} says that, for each $n$,
$$ \LIP (u_n) \leq C\|\Lip_{u_n}\|_\infty,$$
and therefore (\ref{lipunsum}) implies that
\begin{equation}\label{LIPunsum}
\sum_{n=1}^\infty \LIP(u_n) < \infty.
\end{equation}
This, combined with the fact that each $u_n$ has compact support in $\Omega\neq X$, implies that the sum
$$ u = \sum_{n=1}^\infty u_n$$
converges uniformly on compact sets to a Lipschitz function $u\in C_0(\Omega)$. It follows from (\ref{LIPunsum}), (\ref{lipunsum}), and (\ref{Lipsum}) that $u$ satisfies conditions (\ref{alberti3}) and (\ref{alberti4}) of Theorem \ref{alberti}.

To verify (\ref{alberti2}), fix $j\in J$ and observe that, by (\ref{ftildedef}) and (\ref{fnjdef2}),
$$ f_j - \sum_{n=1}^m (d^j u_n ) = f^m_j,$$
almost everywhere in $U_j \cap (\Omega\setminus A)$, for each positive integer $m$. Thus,
\begin{align*}
\|f_j - d^j u\|_{L^\infty(U_j \cap (\Omega\setminus A))} 
&\leq \|f^m_j\|_{L^\infty(U_j \cap (\Omega\setminus A))} + \sum_{n=m+1}^\infty \|d^j u_n\|_{L^\infty(U_j \cap (\Omega\setminus A))}\\
&\leq \alpha_m + \sum_{n=m+1}^\infty \|d^j u_n\|_{L^\infty(U_j \cap (\Omega\setminus A))}\\
&\leq \alpha_m + c_j \sum_{n=m+1}^\infty \|\Lip_{u_n}\|_{L^\infty(U_j \cap (\Omega\setminus A))}
\end{align*}
and both of these tend to zero as $m$ tends to infinity. In the last inequality, we used the fact that $(U_j, \phi_j)$ is a normalized chart, see Definition \ref{normalized}.

Thus,
$$ f_j = d^j u \text{ a.e. on } U_j \cap (\Omega\setminus A),$$
so (\ref{alberti2}) holds. This completes Step 1.

\textbf{Step 2:} The functions $\{f_j\colon U_j\cap \Omega\rightarrow \RR^{k_j}\}$ are arbitrary Borel functions.

We first extend each $f_j$ to be zero off of $U_j$, so that each $f_j$ is defined on all of $\Omega$.

Fix $\epsilon>0$. Choose $r>0$ large so that
$$B = \{x:|f_j(x)|>r \text{ for some } j\in J \}$$
satisfies $\mu(B)<\epsilon/2$. Note that this is possible because, using the fact that $f_j=0$ off $U_j$, we see that
$$\mu\left(\Omega \setminus \bigcup_{\ell=1}^\infty \{x: f_j(x) \leq \ell \text{ for all } j\in J\}\right) =0.$$

For each $j\in J$ let \[
 \tilde{f}_j(x) =
  \begin{cases}
   f_j(x) & \text{if } |f_j(x)|\leq r \\
   rf_j(x)/|f_j(x)|      & \text{if } |f_j(x)|>r.
  \end{cases}
\]
Then $\{\tilde{f}_j\}$ is a uniformly bounded collection of Borel functions on $\Omega$ such that, for all $j\in J$, $|\tilde{f}_j|\leq |f_j|$ everywhere and $\tilde{f}_j=f_j$ outside the set $B$. Fix an open set $A_1 \supseteq B$ such that $\mu(A_1)< \epsilon/2$. Then, for all $j\in J$, $\tilde{f}_j=f_j$ outside of $A_1$.

Now apply the result of Step 1 to the uniformly bounded collection $\{\tilde{f}_j\}$. We obtain an open set $A_2$ with $\mu(A_2)\leq \frac{\epsilon}{2}\mu(\Omega)$ and a Lipschitz function $u\in C_0(\Omega)$ such that 
$$ d^j u = \tilde{f}_j  \text{ a.e. in } U_j \cap (\Omega \setminus A_2),$$
$$ \|\Lip_u\|_p \leq 4C'(\epsilon/2)^{\frac{1}{p}-\frac{1}{\eta}}\left(\sum_{j\in J} \int_{U_j \cap \Omega} |\tilde{f}_j|^p \right)^{1/p} $$
for all $p\in [1,\infty)$, and
$$ \|\Lip_u\|_\infty \leq 4C'(\epsilon/2)^{-\frac{1}{\eta}} \sup_{j\in J} |\tilde{f}_j|_\infty. $$

Thus, for each $j\in J$, $f_j=d^j u$ a.e. in $U_j \cap (\Omega \setminus A)$, where $A=A_1 \cup A_2$ has $\mu(A) \leq \epsilon \mu(\Omega)$. This verifies (\ref{alberti1}) and (\ref{alberti2}).

If $p\in[1,\infty)$, we have
$$ \|\Lip_u\|_p \leq 4C'(\epsilon/2)^{\frac{1}{p}-\frac{1}{\eta}}\left(\sum_{j\in J} \int_{U_j \cap \Omega} |\tilde{f}_j|^p \right)^{1/p} \leq 4C'2^{1/\eta}\epsilon^{\frac{1}{p}-\frac{1}{\eta}}\left(\sum_{j\in J} \int_{U_j \cap \Omega} |f_j|^p \right)^{1/p}, $$
which verifies (\ref{alberti3}). A similar calculation verifies (\ref{alberti4}).

This completes the proof of Theorem \ref{alberti}.
\end{proof}

\section{Proof of Theorem \ref{mpthm}}\label{sectionproof2}
In this section, we give the proof of Theorem \ref{mpthm}. Given our Theorem \ref{alberti}, we can now just closely follow the proof given by Moonens-Pfeffer in \cite{MP08}. For the convenience of the reader, we give most of the details, although they are very similar to those of \cite{MP08}.

In our setting, the analog of Corollary 1.2 in \cite{MP08} is the following:
\begin{lemma}\label{mplemma}
Let $X$ be a PI space with a normalized differentiable structure $(U_j, \phi_j:U_j\rightarrow\RR^{k_j})$. Let $\Omega\subset X$ be a bounded open subset of $X$ and let $\{f_j\colon U_j \cap \Omega \rightarrow \RR^{k_j}\}$ be a collection of Borel functions. Then for every $\epsilon>0$, there exist a compact set $K\subset U$ and a Lipschitz function $u\in C_c(\Omega)$ such that
\begin{equation}\label{mplemma1}
\mu(\Omega\setminus K)< \epsilon,
\end{equation}
\begin{equation}\label{mplemma2}
d^j u = f_j \text{ a.e. in } U_j\cap K
\end{equation}
for each $j\in J$, and
\begin{equation}\label{mplemma3}
|u(x)|\leq \epsilon \min\{1, \dist^2(x, X\setminus \Omega)\}
\end{equation}
for all $x\in X$.
\end{lemma}
\begin{proof}
We can again assume without loss of generality that $\Omega\neq X$, otherwise we replace $\Omega=X$ by $X\setminus\{x_0\}$ for some $x_0\in X$.  Extend the functions $f_j$ to all of $X$ by letting $f_j=0$ off of $U_j\cap \Omega$.

Fix an open set $\Omega'$ compactly contained in $\Omega$ with $\mu(\Omega\setminus \Omega')<\epsilon/2$.

As in Step 2 in the proof of Theorem \ref{alberti}, we can find a compact set $B \subset \Omega'$ such that $\mu(\Omega'\setminus B)<\epsilon/4$ and $\{f_j\}$ are uniformly bounded on $B$, i.e., $\sup_{j\in J} \|f_j\|_{L^\infty(B)} = M < \infty$.

For each $j\in J$, let $g_j = f_j \chi_B$, so the functions $g_j$ are uniformly bounded by the constant $M>0$. Let
$$ \Delta = \min\{1, \dist^2(\Omega', X\setminus \Omega)\}$$
and
$$ d= \epsilon^{1+\frac{1}{\eta}} \Delta / (1+(8\mu(\Omega'))^{\frac{1}{\eta}}CM), $$
where $C$ and $\eta$ are the constants from Theorem \ref{alberti}.

Choose $k$ large so that there are cubes $Q_1, \dots, Q_m\subset \Omega'$ in $\Delta_k$, of diameter at most $d$, that satisfy
$$ \mu(\Omega'\setminus \cup_1^m Q_i)<\epsilon/4. $$
(Note that the doubling property of $\mu$ and the boundedness of $\Omega$ implies that the collection $\{Q_1, \dots, Q_m\}$ really is finite.)

For each $1\leq i \leq m$, we now apply Theorem \ref{alberti} to the collection $\{g_j\}$ in the cube $Q_i$ with parameter $\epsilon' = \epsilon/8\mu(\Omega')$. For each $1\leq i \leq m$, we obtain a compact set $K_i \subset Q_i$ with $\mu(Q_i\setminus K_i)\leq \epsilon'\mu(Q_i)$ and a Lipschitz function $u_i\in C_c(Q_i)$ such that, for each $j\in J$, $d^j u_i = g_j$ almost everywhere in $U_j \cap K_i$.

Furthermore, the remark after the statement of Theorem \ref{alberti} shows that
$$ \LIP u_i \leq C(\epsilon')^{-\frac{1}{\eta}}M.$$

As $u_i\in C_c(Q_i)$, it follows that, for each $1\leq i\leq m$,
$$ \|u_i\|_\infty \leq (\diam Q_i) \LIP(u_i) \leq d C(\epsilon')^{-\frac{1}{\eta}}M < \epsilon\Delta. $$

Let $K = B \cap \left(\cup_{i=1}^m K_i\right)$, a compact subset of $\Omega$. Our choices easily imply that
$$ \mu(\Omega\setminus K)< \epsilon,$$
which verifies (\ref{mplemma1}).

Let $u=\sum_{i=1}^m u_i$. Then $u$ is a Lipschitz function in $C_c(\Omega)$ that satisfies
$$ d^j u = f_j \text{ almost everywhere in } U_j \cap K,$$
so (\ref{mplemma2}) holds.

To verify the final condition, note that $u$ is identically zero outside of $\Omega'$, so (\ref{mplemma3}) holds there automatically. For $x\in \Omega'$, we have
$$ |u(x)| \leq \sup_i \|u_i\|_\infty < \epsilon \Delta  \leq \epsilon \min\{1, \dist^2(x, X\setminus \Omega)\}. $$
Thus, the final condition (\ref{mplemma3}) of Lemma \ref{mplemma} is verified.
\end{proof}

We now prove Theorem \ref{mpthm}. (To avoid some cumbersome subscripts, we change notation slightly and write $\Lip(g)(x)$ instead of $\Lip_g(x)$.)

\begin{proof}[Proof of Theorem \ref{mpthm}]
We again closely follow \cite{MP08}.

By Lemma \ref{WLOGlemma}, we may assume that the measurable differentiable structure is normalized. Without loss of generality, we also assume that $\epsilon <1$. Fix $x_0\in X$ and let $B_i = B(x_0, i)$ for each $i\in \mathbb{N}$.

We repeatedly apply Lemma \ref{mplemma}. We inductively construct compact sets $K_i\subset \Omega_i = \Omega\cap B_i \setminus \cup_{k=1}^{i-1} K_k$ and Lipschitz functions $u_i\in C_c(\Omega_i)$ such that, for each $i\in \mathbb{N}$
\begin{equation}\label{mplemmause1}
\mu(\Omega_i \setminus K_i) < 2^{-i}\epsilon < 2^{-i},
\end{equation}
\begin{equation}\label{mplemmause2}
d^j u_i = f_j - \sum_{k=1}^{i-1} d^j u_k \text{ a.e. in } K_i,
\end{equation}
and
\begin{equation}\label{mplemmause3}
|u_i(x)|\leq 2^{-i}\epsilon \min\{1, \dist^2(x,X\setminus \Omega_i)\}
\end{equation}
for all $x\in X$.

Let $K= \cup_{i=1}^\infty K_i$ and let $u = \sum_{i=1}^\infty u_i$. Note that $u$ is a continuous function, because the bound $\|u_i\|_\infty\leq 2^{-i}\epsilon$ from (\ref{mplemmause3}) implies the uniform convergence of this sum. It also implies that $\|u\|_\infty \leq \epsilon$, verifying the first part of (\ref{mpthm1}).

The second part of (\ref{mpthm1}) also follows immediately, by observing that
$$ \{u\neq 0\} \subset \bigcup_{i=1}^\infty \{u_i\neq 0\} \subset \bigcup_{i=1}^\infty \Omega_i \subset \Omega.$$

In addition, $\mu\left((\Omega \cap B_i) \setminus K\right) \leq \mu(\Omega_k \setminus K_k)\leq 2^{-k}$ whenever $k\geq i$, which implies that $\mu(\Omega \cap B_i \setminus K) = 0$ for each $i\in\mathbb{N}$ and thus that $\mu(\Omega \setminus K)=0$.

It remains to verify \ref{mpthm2} and \ref{mpthm3}.

We first claim that if $x\in K_i$ and $k>i$, then
\begin{equation}\label{Lipk>i}
\Lip\left(\sum_{k=i+1}^\infty u_k\right)(x) = 0.
\end{equation}
Indeed, note that for $k>i$ and $x\in K_i$, we have $K_i \cap \Omega_k = \emptyset$ and so $u_k(x)=0$. Fix any $y\in X$. If $y\notin \Omega_k$, then $u(y)=0$ as well. If $y\in \Omega_k$, then
$$ |u_k(y) - u_k(x)| = |u_k(y)| \leq 2^{-i}\epsilon d(x,y)^2 $$
by (\ref{mplemmause3}). So, in either case, we have 
$$ |u_k(y) - u_k(x)| \leq 2^{-i}\epsilon d(x,y)^2 $$
whenever $x\in K_i$, $y\in X$, and $k>i$. Summing this over all $k>i$ immediately proves (\ref{Lipk>i}).

Therefore, for almost every $x\in K_i\cap U_j$, we have the following: 
\begin{align*}
\Lip(u - f_j(x)\cdot \phi_j)(x) &\leq \Lip\left(\sum_{k=1}^i u_k - f_j(x)\cdot \phi_j\right)(x)\\
&\leq \Lip\left(\sum_{k=1}^i u_k - \left(\sum_{k=1}^i d^j u_k(x)\right)\cdot \phi_j\right)(x)\\
&= 0.
\end{align*}
It follows that at almost every point in $K_i \cap U_j$, the function $u$ is differentiable with $d^j u = f_j$. Because $\mu(\Omega\setminus \cup K_i)=0$, it follows that $d^j u = f_j$ almost everywhere in $\Omega \cap U_j$. This proves (\ref{mpthm2}).

Finally, we must show (\ref{mpthm3}), that $\Lip_u = 0$ everywhere in $X\setminus \Omega$. Fix $x\in X\setminus \Omega$. If $y\in X\setminus \Omega$, then $u(x)= u(y) = 0$.

If $y\in \Omega$, then
$$|u(y) - u(x)| = |u(y)| \leq \epsilon \dist^2(x, X\setminus \Omega) \leq \epsilon d(x,y)^2.$$
Thus, for any $x\in X\setminus \Omega$ and any $y\in X$, we have
$$|u(y) - u(x)|\leq \epsilon d(x,y)^2,$$
which immediately implies that $\Lip_u(x)=0$. This completes the argument.
\end{proof}

\bibliography{unifiedbib}{}
\bibliographystyle{plain}

\end{document}